\documentclass[12pt]{amsart}
\usepackage{enumerate}   
\usepackage[margin=0.9in]{geometry}
\usepackage{amsmath,amssymb,amsthm}
\usepackage{color}
\newcommand{\re}{\mathbb{R}}
\newcommand{\co}{\mathbb{C}}

\newcommand{\cc}{\mathcal{C}}

\newcommand{\z}{\bar z}

\newcommand{\rp}{\mbox{Re}}

\newcommand{\boxx}{\rule{2.12mm}{3.43mm}}

\title{Removing Isolated Zeroes by Homotopy}

\author[Adam Coffman --- Ji\v r\'\i\ Lebl]{Adam Coffman --- Ji\v r\'\i\ Lebl}

\address
    {\textsc{Adam Coffman}\\ Department of Mathematical Sciences\\
  Purdue University Fort Wayne\\ Fort Wayne, IN 46805, USA}
\email{CoffmanA@pfw.edu}

\address{\textsc{Ji\v r\'\i\ Lebl}\\ Department of Mathematics\\ Oklahoma
  State University\\ Stillwater, OK 74078, USA}
\email{lebl@math.okstate.edu}

\thanks{The second author was in part supported by NSF grant DMS-1362337.}

\subjclass[2010]{Primary 57R45; Secondary 14P10, 31B25, 35J25, 41A29, 57R25,
  58K25, 58K45, 58K60}

\keywords{Isolated zero; semialgebraic map; singularities of
  differentiable mappings}

\newtheorem{thm}{Theorem}[section]
\newtheorem{prop}[thm]{Proposition}
\newtheorem{lem}[thm]{Lemma}
\newtheorem{cor}[thm]{Corollary}

\theoremstyle{definition}

\newtheorem{defn}[thm]{Definition}
\newtheorem{notation}[thm]{Notation}
\newtheorem{example}[thm]{Example}

\newtheorem{question}[thm]{Question}

\theoremstyle{remark}

\newtheorem{rem}[thm]{Remark}

\begin{document}

\begin{abstract}
Suppose that the inverse image of the zero vector by a continuous map
$f:{\mathbb R}^n\to{\mathbb R}^q$ has an isolated point $P$.  The
existence of a continuous map $g$ which approximates $f$ but is
nonvanishing near $P$ is equivalent to a topological property we call
``locally inessential,'' generalizing the notion of index zero for
vector fields, the $q=n$ case. We consider the problem of constructing
such an approximation $g$ and a continuous homotopy $F(x,t)$ from $f$
to $g$ through locally nonvanishing maps.  If $f$ is a semialgebraic
map, then there exists $F$ also semialgebraic.  If $q=2$ and $f$ is
real analytic with a locally inessential zero, then there exists a
H\"older continuous homotopy $F(x,t)$ which, for $(x,t)\ne(P,0)$, is
real analytic and nonvanishing.  The existence of a smooth homotopy,
given a smooth map $f$, is stated as an open question.
\end{abstract}

\maketitle

\section{Introduction}\label{sec0b}

For a continuous vector field on a manifold, it is well-known that an
isolated zero can be removed by a small, local perturbation if and
only if that zero has an ``index'' equal to $0$.  That is, for
a vector field $\mathbf f$ vanishing with index $0$ at $\vec p$, 
and any small neighborhood of $\vec p$, there
is another vector field $\mathbf g$ agreeing with $\mathbf f$ outside that neighborhood, and arbitrarily
$\cc^0$-close to $\mathbf f$ but nonvanishing inside it.  In fact, the zero
is removable in the following stronger, but less well-known, sense (\cite{d84}):
not only is there such a perturbation $\mathbf g$, there is a continuous 
homotopy $F(\vec x,t)$ from $\mathbf
f$ to $\mathbf g$ such that $F(\vec x,t)$ is nonvanishing for $\vec x$
near $\vec p$ and $t>0$.  So, the isolated zero can be removed
instantaneously.

We consider the generalization of this phenomenon to other dimensions:
locally, maps ${\mathbf f}:{\mathbb R}^n\to{\mathbb R}^q$, but our
primary interest is in the regularity of the homotopy $F$.  One
version of our Main Question \ref{question4.3} asks: given that
$\mathbf f$ is smooth, does there exist a smooth homotopy $F$ that
instantly removes an isolated zero, assuming only that there is no
topological obstruction to a nonvanishing approximation?  This remains
open; as remarked by \cite{d83b}, \cite{d84}, just $\cc^1$ regularity
for $F$ seems to be a difficult question even under strong assumptions
about $\mathbf f$.  Our main results consider cases where $\mathbf f$
is semialgebraic (Theorem \ref{thm0.7}) or real analytic (Theorem
\ref{thm5.1}).

When the target dimension $q$ is not equal to $n$, isolated zeros are
no longer a generic phenomenon, but in Section \ref{sec2}, we show
that with a natural generalization of the notion of ``index zero''
(Definition \ref{def3.2}), the analogous nonvanishing approximation
property holds.  Instead of considering $\mathbf f$ as a vector field,
another way to visualize an isolated zero of $\mathbf f:\re^n\to\re^q$
is to consider the zero sets of its components $(f_1(\vec
x),\ldots,f_q(\vec x))$.  For $q$ suitably smooth and generic
functions, each component vanishes on some hypersurface, and for $q\le
n$, the intersection of $q$ hypersurfaces in general position is
expected to be a set with dimension $n-q$.  So for $q<n$, an isolated
point in the intersection indicates that the hypersurfaces are not in
general position, but our question is about the persistence of the
isolated zero: is there some perturbation so that the isolated point
of intersection disappears, or do the $q$ hypersurfaces continue to
have a non-empty, but not necessarily isolated, intersection after any
perturbation?  In the case $n=3$, $q=2$ (of special interest for
applications such as computer graphics), can two real implicit
surfaces $\{f_1=0\}$, $\{f_2=0\}$, meeting at just one point (e.g.,
two cones sharing a vertex; for another example, see Section
\ref{sec5}), always be made disjoint by small changes in $f_1$ and
$f_2$?  The answer is yes, and ``small changes'' can be interpreted as
either a local $\cc^0$ approximation, or a continuous family of such
approximations (the homotopy $F$).

Our first approach to the construction of $F$ is to start with the
local approximation; the results in Section \ref{sec2} generalize
well-known facts about the vector field case ($q=n$).  Then in Section
\ref{sec3} we construct the continuous homotopy; the $q=n$ case was
considered by \cite{d84}; our construction for any $q$ in Theorem
\ref{thm0.7} is explicit enough so that the homotopy $F$ is
semialgebraic if $\mathbf f$ is.

The second approach, in Section \ref{sec4}, considers the case where
$q=2$ and $\mathbf f$ is real analytic; the construction of Theorem
\ref{thm5.1} uses PDE methods (only the classical Dirichlet problem,
with H\"older estimates up to the boundary) to construct a continuous
nonvanishing homotopy $F$ which is real analytic except at the point
$(\vec p,0)$, near which it satisfies an inequality of the form
$\|F(\vec x,t)\|\le C\|(\vec x-\vec p,t)\|$.  Either better regularity
or a counterexample would be interesting: is there some polynomial map
with an isolated zero that can be removed by a semialgebraic homotopy
but not by a $\cc^1$ homotopy?

Time-dependent vector fields are of obvious importance in
applications, and the behavior of their zeros remains a topic of
current interest (\cite{NNPV}).  We further remark that the global
problem of finding a homotopy from a map $\mathbf f$ to another map
with fewer zeros has been considered in algebraic topology
(\cite{bfgj}, \cite{f}).  The topic of ``root theory'' is a special
case of the coincidence problem of finding homotopies from two maps
$\mathbf f$ and $\mathbf g$ at time $0$ to two other maps at time $1$
with disjoint images (or with a minimal number of points of
intersection).  Our results in Sections \ref{sec3} and \ref{sec4} are
different, in that we want to find a homotopy where the image of
$\mathbf f$ becomes disjoint from $\{\mathbf0\}$ for all $t>0$.

Finally, we mention that our interest in this topic started with an
analogue in CR geometry --- \cite{e2} considers a real $3$-manifold
embedded in $\co^3$ with an isolated complex tangent, and describes a
local topological obstruction to the existence of a $\cc^1$ homotopy
to a totally real embedding.

\section{Notation}\label{sec1}

Fix a positive integer $n$.  We are interested in maps $\mathbf f$
from an $n$-manifold to another manifold, where the inverse image of a
point $O$ contains an isolated point.  Our analysis is local, not
global, so we can consider the target manifold to be $\re^q$ for some
$q$ and the point $O$ to be the zero vector $\mathbf0$; then an
isolated point of ${\mathbf f}^{-1}(\{O\})$ is called an isolated zero
of $\mathbf f$ (Definition \ref{def3.1}).  We consider the domain of
$\mathbf f$ and its perturbations to be a neighborhood $\Omega$ in
$\re^n$ rather than a more general $n$-manifold.

\begin{notation}\label{not0.1}
  For $\vec c\in\re^n$ and $R>0$, the following notation is used for
  the standard Euclidean balls and spheres with center $\vec c$ and
  radius $R$:
\begin{eqnarray*}
  B^n(\vec c,R)&=&\{\vec x:\|\vec x-\vec c\|<R\}=\mbox{ the open ball}\\
  {\overline{B}}^n(\vec c,R)&=&\{\vec x:\|\vec x-\vec c\|\le R\}=\mbox{ the closed ball}\\
  B^n_*(\vec c,R)&=&\{\vec x:0<\|\vec x-\vec c\|<R\}=\mbox{ the punctured open ball}\\
  {\overline{B}}^n_*(\vec c,R)&=&\{\vec x:0<\|\vec x-\vec c\|\le R\}=\mbox{ the punctured closed ball}\\
  S^{n-1}(\vec c,R)&=&\{\vec x:\|\vec x-\vec c\|=R\}=\mbox{ the boundary sphere}
\end{eqnarray*}
  Similarly, for $\re^{n+1}$ with coordinates $(\vec x,t)$, denote the upper half-ball and upper hemisphere:
\begin{eqnarray*}
  B^{n+1}_+((\vec c,\tau),R)&=&\{(\vec x,t):\|(\vec x,t)-(\vec c,\tau)\|<R, t>\tau\}\\
  S^{n}_+((\vec c,\tau),R)&=&\{(\vec x,t):\|(\vec x,t)-(\vec c,\tau)\|=R, t>\tau\}
\end{eqnarray*}
  The usual unit sphere $S^{q-1}(\vec0,1)$ in $\re^q$ will be
  abbreviated $S^{q-1}$.  The restriction of a map ${\mathbf
    f}:\Omega\to\re^q$ to some sphere $S^{n-1}(\vec
  c,r)\subseteq\Omega$ will be denoted ${\mathbf f}|_S$.
\end{notation}
\begin{notation}\label{not0.2}
  For a map ${\mathbf f}:\Omega\to\re^q$
  with components ${\mathbf f}(\vec x)=(f_1(\vec x),\ldots,f_q(\vec
  x))$, denote the locus: $$V({\mathbf f})=V(f_1,\ldots,f_q)=\{\vec
  x\in \Omega:f_1(\vec x)=0,\ldots,f_q(\vec x)=0\}={\mathbf
  f}^{-1}(\{{\mathbf0}\}).$$
\end{notation}
\begin{defn}\label{def3.1}
  For a function ${\mathbf f}:\Omega\to\re^q$ and a point $\vec
  p\in\Omega$ such that ${\mathbf f}(\vec p)={\mathbf0}$, $\vec p$ is
  an {\underline{isolated}} zero of $\mathbf f$ means that there
  exists some $R>0$ so that $B^n(\vec p,R)\subseteq\Omega$ and:
  \begin{equation*}\label{eq1}
    V({\mathbf f})\cap B^n(\vec p,R)=\{\vec p\}.
  \end{equation*}
\end{defn}
\begin{defn}\label{def1.3}
  A subset $\Sigma\subseteq\re^N$ is a {\underline{semialgebraic set}}
  means that $\Sigma$ is a finite union of sets of the form $$\{\vec
  x\in\re^N:P(\vec x)=0,Q_1(\vec x)>0,\ldots,Q_j(\vec x)>0\},$$ where
  each of the functions $P$, $Q_1$, \ldots, $Q_j$ is a polynomial
  $\re^N\to\re^1$.  For an open domain $\Omega\subseteq\re^n$ as above, a function ${\mathbf f}:\Omega\to\re^{q}$ is a
  {\underline{semialgebraic map}} means that the graph of $\mathbf f$
  is a semialgebraic set in $\re^{n+q}$.
\end{defn}
\begin{rem}\label{rem1.4}
  Some references require that semialgebraic maps are also continuous;
  it is more convenient for us to instead allow discontinuous maps and
  explicitly mention continuity when needed.
\end{rem}
\begin{prop}\label{prop1.4}
  A linear projection of a semialgebraic set is a semialgebraic set.
  The scalar valued semialgebraic maps $f:\re^{n}\to\re^1$ form a ring.  A map
  ${\mathbf f}:\re^n\to\re^q$ is semialgebraic if and only if its
  components, ${\mathbf f}=(f_1,\ldots,f_q)$, are semialgebraic
  $f_j:\re^n\to\re^1$.  The composite of semialgebraic maps is
  semialgebraic.  \boxx
\end{prop}
\begin{rem}\label{rem1.6}
  The notions of Definition \ref{def1.3} and the claims of Proposition
  \ref{prop1.4} are well-known (we refer to \cite{bcr}).  For an open domain $\Omega$,
  it follows
  from Definition \ref{def1.3} and the projection property that if ${\mathbf f}:\Omega\to\re^q$ is
  a semialgebraic map, then $\Omega$ must be an open semialgebraic
  set.
\end{rem}

\section{Nonvanishing approximation}\label{sec2}

Theorem \ref{thm0.3} and its Corollaries are about approximating a map
with an isolated zero by a nonvanishing map.  The case $q=n$ is the
well-known situation of a vector field with index zero, for which
\cite{st}, \cite{ankerthesis}, \cite{anker}, prove results similar to
Theorem \ref{thm0.3} and additionally give estimates for derivatives
(see also \cite{pp}).  For $q\le n$, the connection between zero sets of
perturbations and homotopy classes of maps of spheres was used
by \cite{d83} to study the stability of non-isolated zero sets, rather
than the removability of isolated zeros.

Definition \ref{def3.2} generalizes the notion of ``index zero'' to
other dimensions.  For this Section, fix a positive integer $q$.
\begin{lem}\label{lem3.2}
  For a continuous map ${\mathbf f}:\Omega\to\re^q$, and an isolated
  zero $\vec p\in B^n(\vec p,R)$ as in Definition \ref{def3.1}, the
  following are equivalent:
  \begin{enumerate}[\rm (i)]
    \item There exists $r\in(0,R)$ such that the restriction of
      ${\mathbf f}$ to the domain $S^{n-1}(\vec p,r)$ and target
      $\re^q\setminus\{{\mathbf0}\}$ is null-homotopic;
    \item For any $r\in(0,R)$, the restriction of ${\mathbf f}$ to the
      domain $S^{n-1}(\vec p,r)$ and target $\re^q\setminus\{{\mathbf0}\}$
      is null-homotopic;
          \item For any $\epsilon>0$ there is some $\delta$ so that
            $0<\delta\le R$ and for any $r\in(0,\delta)$, the
            restriction of ${\mathbf f}$ to the domain $S^{n-1}(\vec p,r)$
            and target $B^q_*({\mathbf0},\epsilon)$ is null-homotopic;
            \item For any $\epsilon>0$ there is some $\delta$ so that
              $0<\delta\le R$ and for any $r\in(0,\delta)$, the
              restriction of ${\mathbf f}$ to the domain $B^n_*(\vec
              p,r)$ and target $B^q_*({\mathbf0},\epsilon)$ is
              null-homotopic.
  \end{enumerate}
\end{lem}
\begin{proof}
  The implication (iv)$\implies$(i) is obvious.  The statement of
  Property (i) is the most easily checked, while the statements of
  Properties (iii) and (iv) are local both in the domain and the target,
  and could be adapted for isolated roots of a map from one manifold
  to another.

  To show that (i)$\implies$(ii), let $r_0$ be the radius from
  Property (i), so that there exists a homotopy $\Phi:S^{n-1}(\vec
  p,r_0)\times[0,1]\to\re^q\setminus\{{\mathbf0}\}$ from $\Phi(\vec
  x,0)={\mathbf f}|_S(\vec x)$ to a constant map $\Phi(\vec
  x,1)\equiv{\mathbf c}$ in $\re^q\setminus\{{\mathbf0}\}$.  For any
  $r\in(0,R)$, define $\Phi_r:S^{n-1}(\vec
  p,r)\times[0,1]\to\re^q\setminus\{{\mathbf0}\}$ by the formula
  $$\Phi_r(\vec x,t)=\left\{\begin{array}{cc}{\mathbf f}(\vec
  p+\left(1+2\left(\frac{r_0}{r}-1\right)t\right)\cdot(\vec x-\vec
  p))&0\le t\le\frac12\\\Phi(\vec p+\frac{r_0}r\cdot(\vec x-\vec
  p),2t-1)&\frac12\le t\le1\end{array}\right\}.$$ Then $\Phi_r$ is a
  continuous homotopy from $\Phi_r(\vec x,0)={\mathbf f}|_S(\vec x)$ to
  the constant map $\Phi_r(\vec x,1)\equiv{\mathbf c}$ in
  $\re^q\setminus\{{\mathbf0}\}$.

  Assuming (ii), for any $\epsilon>0$ there is, by the continuity of
  $\mathbf f$, some $\delta$, $0<\delta\le R$, such that if $\|\vec
  x-\vec p\|<\delta$ then $\|{\mathbf f}(\vec x)-{\mathbf f}(\vec
  p)\|<\epsilon$.  For any $r\in(0,\delta)$, the image ${\mathbf
  f}(S^{n-1}(\vec p,r))$ is contained in $B_*^q({\mathbf0},\epsilon)$.
  By Property (ii), ${\mathbf f}|_S$ is homotopic in
  $\re^q\setminus\{{\mathbf0}\}$ to a constant map; the claim of
  Property (iii) is just that it is also homotopic in
  $B_*^q({\mathbf0},\epsilon)$ to a (possibly different) constant map.

  Let $\Phi_r:S^{n-1}(\vec p,r)\times[0,1]\to\re^q\setminus\{{\mathbf0}\}$
  be the homotopy from Property (ii), so that $\Phi_r(\vec
  x,0)={\mathbf f}|_S(\vec x)$ and $\Phi_r(\vec
  x,1)\equiv{\mathbf{c}_r}\in\re^q\setminus\{{\mathbf0}\}$.  By the
  continuity of $\Phi_r$ and the compactness of $S^{n-1}(\vec p,r)$, there
  is some $\delta_2$, $0<\delta_2\le1$, so that the image
  $\Phi_r(S^{n-1}(\vec p,r)\times[0,\delta_2])$ is contained in
  $B^q_*({\mathbf0},\epsilon)$.  Also, $\Phi_r$ achieves a maximum
  magnitude $M=\max\{\|\Phi_r(\vec x,t)\|:(\vec x,t)\in S^{n-1}(\vec
  p,r)\times[0,1]\}>0$.  If $\delta_2=1$ or $M<\epsilon$, Property (iii)
  already holds; otherwise, for $0<\delta_2<1$ and $M\ge\epsilon$,
  define the following weakly decreasing, continuous function
  $\gamma:[0,1]\to(0,1]$:
  $$\gamma(t)=\left\{\begin{array}{cc}\frac{\frac{\epsilon}{2M}-1}{\delta_2}t+1&0\le
    t\le\delta_2\\\frac{\epsilon}{2M}&\delta_2\le
    t\le1\end{array}\right\}.$$ Then $\Psi(\vec
    x,t)=\gamma(t)\cdot\Phi_r(\vec x,t)$ is a continuous homotopy from
    $\Psi(\vec x,0)=1\cdot\Phi_r(\vec x,0)={\mathbf f}(\vec x)$ to the
    constant map $\Psi(\vec
    x,1)=\frac{\epsilon}{2M}\cdot{\mathbf{c}_r}\in
    B^q_*({\mathbf0},\epsilon)$.  For $0\le t\le\delta_2$,
    $\|\Psi(\vec x,t)\|=\gamma(t)\|\Phi_r(\vec x,t)\|<1\cdot\epsilon$,
    and for $\delta_2\le t\le 1$, $\|\Psi(\vec
    x,t)\|=\frac{\epsilon}{2M}\|\Phi_r(\vec
    x,t)\|\le\frac{\epsilon}{2}$, which establishes Property (iii).

  Assuming (iii), for any $\epsilon>0$ there is, by the continuity of
  $\mathbf f$, some $\delta_1$, $0<\delta_1\le R$, such that if
  $\|\vec x-\vec p\|<\delta_1$ then $\|{\mathbf f}(\vec x)-{\mathbf
    f}(\vec p)\|<\epsilon$.  There is, from Property (iii), some
  $\delta_2>0$ corresponding to the same $\epsilon$; the claimed
  $\delta$ will be $\min\{\delta_1,\delta_2\}$.  Let $r$ be any radius
  in $(0,\delta)$, and then there is a homotopy $\Psi_r:S^{n-1}(\vec
  p,r)\times[0,1]\to B_*^q({\mathbf0},\epsilon)$ from ${\mathbf
    f}|_S$ to a constant ${\mathbf c}$ in
  $B_*^q({\mathbf0},\epsilon)$. The map $\Theta:B_*^n(\vec
  p,r)\times[0,1]\to B_*^q({\mathbf0},\epsilon)$, defined by the
  following formula, is a homotopy in $B_*^q({\mathbf0},\epsilon)$ as
  claimed in Property (iv):
      $$\Theta(\vec x,t)=\left\{\begin{array}{cc}{\mathbf f}(\vec
  p+\left(\frac{2(r-\|\vec x-\vec p\|)t+\|\vec x-\vec p\|}{\|\vec
    x-\vec p\|}\right)\cdot(\vec x-\vec p))&0\le
  t\le\frac12\\\Psi_r(\vec p+\frac{r}{\|\vec x-\vec p\|}\cdot(\vec
  x-\vec p),2t-1)&\frac12\le t\le1\end{array}\right\}.$$
\end{proof}
\begin{defn}\label{def3.2}
  For a continuous map ${\mathbf f}:\Omega\to\re^q$, a point $\vec p$
  such that ${\mathbf f}(\vec p)={\mathbf0}$ is a {\underline{locally
      inessential}} zero of $\mathbf f$ means: $\vec p$ is an isolated
  zero and any of the equivalent properties of Lemma \ref{lem3.2} is
  satisfied.
\end{defn}
\begin{rem}\label{rem3.4}
  For the case $q=n$, the above notion is exactly that the isolated
  zero of $\mathbf f$ has Poincar\'e-Hopf index $0$.  The new term
  ``locally inessential zero'' could be replaced by just using
  ``isolated zero with index $0$'' more generally for any $q$, $n$,
  but such usage would raise the unrelated (for us) question of
  whether or how one could define a non-zero index for $q\ne n$.
\end{rem}

\begin{lem}\label{lem3.5}
  If $n$ and $q$ satisfy \begin{equation*}
    \pi_{n-1}(S^{q-1})\cong\{0\},
  \end{equation*}
  then any isolated zero of any continuous map ${\mathbf
    f}:\Omega\to\re^q$ is a locally inessential zero.
\end{lem}  
\begin{proof}
  For $\vec p$ and $R>0$ as in Definition \ref{def3.1}, let $r$ be any
  radius in $(0,R)$; we will check Property (i) from Lemma
  \ref{lem3.2}.  The space $\re^q\setminus\{\mathbf0\}$ is
  homeomorphic to the product $S^{q-1}\times(0,1)$, so its
  $(n-1)^{th}$ homotopy group is the same as that of $S^{q-1}$, and
  trivial by hypothesis.  Any continuous map from an $(n-1)$-sphere to
  $\re^q\setminus\{\mathbf0\}$, including the restriction ${\mathbf
  f}|_S$ from Property (i), is homotopic to a constant map.
\end{proof}
\begin{rem}\label{rem3.6}
  The hypothesis on $n$ and $q$ in Lemma \ref{lem3.5} is satisfied for
  $q=2$ and $n\ge3$, or for any pair where $n<q$.
\end{rem}

\begin{thm}\label{thm0.3}
  Let ${\mathbf f}:\Omega\to\re^q$ be a continuous map with an
  isolated zero at $\vec p$ and $R>0$ as in Definition \ref{def3.1}.
  The following are equivalent:
  \begin{itemize}
  \item $\vec p$ is a locally inessential zero of $\mathbf f$;
    \item For any $\epsilon>0$ and any $\rho$ with $0<\rho\le R$,
      there exists a continuous map ${\mathbf g}:\Omega\to\re^q$ such
      that:
  \begin{eqnarray}
    \vec x\in\Omega\setminus B^n(\vec p,\rho)\implies {\mathbf g}(\vec
    x)&=&{\mathbf f}(\vec x),\mbox{\ \ and}\label{eq21}\\ \vec
    x\in\Omega\implies\|{\mathbf g}(\vec x)-{\mathbf f}(\vec
    x)\|&<&\epsilon,\mbox{\ \ and}\label{eq20}\\ V({\mathbf g})\cap
    B^n(\vec p,R)&=&\mbox{\rm{\O}}.\label{eq22}
  \end{eqnarray}
  \end{itemize}
\end{thm}
\begin{proof}
  First, assume that $\vec p$ is a locally inessential zero.  Given
  $\epsilon>0$, there is some $\delta_1>0$ corresponding to
  $\frac{\epsilon}2$ in Property (iii) from Lemma \ref{lem3.2}.  Let
  $\delta=\min\{\delta_1,\rho\}$, and denote by ${\mathbf f}|_S$ the
  restriction of $\mathbf f$ to the domain $S^{n-1}(\vec p,\frac{\delta}2)$
  and target $B^q_*({\mathbf0},\frac{\epsilon}2)$.

  By Property (iii) from Lemma \ref{lem3.2}, there exists a homotopy
  \begin{equation}\label{eq0}
    \varphi:S^{n-1}(\vec p,\frac{\delta}2)\times[0,1]\to B^q_*({\mathbf0},\frac{\epsilon}2)
  \end{equation}
  from $\varphi(\vec x,0)={\mathbf f}|_S(\vec x)$ to a constant map
  $\varphi(\vec x,1)\equiv{\mathbf c}$ in
  $B^q_*({\mathbf0},\frac{\epsilon}2)$.  By compactness, $\varphi$
  achieves some minimum magnitude $0<m\le\|\varphi(\vec
  x,t)\|<\frac{\epsilon}2$.  Now ${\mathbf f}|_S$ extends to a
  continuous ${\mathbf {\tilde f}}:\overline{B}^n(\vec
  p,\frac{\delta}2)\to B^q_*({\mathbf0},\frac{\epsilon}2)$.  In the
  interest of giving explicit formulas when we can, the extension
  constructed in \cite{s} \S1.3 is defined by:
  \begin{equation}\nonumber
    {\mathbf {\tilde f}}(\vec x)=\left\{\begin{array}{cc}{\mathbf
    c}&0\le\|\vec x-\vec p\|\le\frac{\delta}4\\\varphi\left(\vec
    p+\frac{\delta}{2\|\vec x-\vec p\|}(\vec x-\vec
    p),2-\frac4{\delta}\|\vec x-\vec
    p\|\right)&\frac{\delta}4\le\|\vec x-\vec
    p\|\le\frac{\delta}2\end{array}\right\},
  \end{equation}
  which is continuous because its two pieces are continuous on closed
  sets and agree on their intersection.

  By construction, for all $\vec x\in\overline{B}^n(\vec
  p,\frac{\delta}2)$, $0<m\le\|{\mathbf {\tilde f}}(\vec
  x)\|<\frac{\epsilon}2$, and for all $\vec x\in S^{n-1}(\vec
  p,\frac{\delta}2)$, ${\mathbf {\tilde f}}(\vec x)={\mathbf
  f}|_S(\vec x)={\mathbf f}(\vec x)$.  Define ${\mathbf
  g}:\Omega\to\re^q$ by ${\mathbf g}(\vec x)={\mathbf f}(\vec x)$ for
  $\|\vec x-\vec p\|\ge\frac{\delta}2$, and ${\mathbf g}(\vec
  x)={\mathbf {\tilde f}}(\vec x)$ for $\|\vec x-\vec
  p\|\le\frac{\delta}2$; then $\mathbf g$ is continuous on $\Omega$,
  and satisfies (\ref{eq21}) and (\ref{eq22}) as claimed.  Further,
  for any $\vec x\in\Omega$, either $\|{\mathbf g}(\vec x)-{\mathbf
  f}(\vec x)\|=0$ or
  \begin{equation*}
  \|{\mathbf g}(\vec x)-{\mathbf f}(\vec x)\|=\|{\mathbf
  {\tilde f}}(\vec x)-{\mathbf f}(\vec x)\|\le\|{\mathbf {\tilde f}}(\vec
  x)\|+\|{\mathbf f}(\vec x)\|<\frac{\epsilon}2+\frac{\epsilon}2=\epsilon.
  \end{equation*}

  Conversely, for any $\epsilon>0$, there is some $\delta$,
  $0<\delta\le R$, so that if $\|\vec x-\vec p\|<\delta$ then
  $\|{\mathbf f}(\vec x)\|<\frac{\epsilon}2$.  Let
  $\rho=\frac{\delta}2$, so there is a continuous $\mathbf g$ with
  $\|{\mathbf g}(\vec x)-{\mathbf f}(\vec x)\|<\frac{\epsilon}2$ for
  all $\vec x\in\Omega$, and ${\mathbf g}(\vec x)={\mathbf f}(\vec x)$
  for all $\vec x\in\Omega\setminus B^n(\vec p,\rho)$.  Let ${\mathbf
  f}|_S$ denote the restriction of $\mathbf f$ to the domain
  $S^{n-1}(\vec p,\rho)$ and target $B^q_*({\mathbf0},\epsilon)$.
  Define $\Psi_\rho:S^{n-1}(\vec
  p,\rho)\times[0,1]\to\re^q\setminus\{{\mathbf0}\}$ by
  $$\Psi_\rho(\vec x,t)={\mathbf g}(\vec p+(1-t)(\vec x-\vec p)).$$
  So, $\Psi_\rho(\vec x,0)={\mathbf g}|_S(\vec x)={\mathbf f}|_S(\vec
  x)$, $\Psi_\rho(\vec x,1)\equiv{\mathbf g}(\vec p)$, and
  \begin{eqnarray*} \|\Psi_\rho(\vec x,t)\|&=&\|{\mathbf g}(\vec
  p+(1-t)(\vec x-\vec p))\|\\ &\le&\|{\mathbf g}(\vec p+(1-t)(\vec
  x-\vec p))-{\mathbf f}(\vec p+(1-t)(\vec x-\vec p))\|+\|{\mathbf
  f}(\vec p+(1-t)(\vec x-\vec p))\|\\
  &<&\frac{\epsilon}2+\frac{\epsilon}2=\epsilon,
  \end{eqnarray*}
  so Property (iii) from Lemma \ref{lem3.2} holds.
\end{proof}

\begin{cor}\label{cor0.4b}
  Let $q$, ${\mathbf f}$, $\vec p$, and $R$ be as in Theorem
  \mbox{\rm{\ref{thm0.3}}}.  If $\vec p$ is a locally inessential zero
  of $\mathbf f$, and, additionally, $\mathbf f$ is smooth, then there
  exists $\mathbf g$ as in Theorem \mbox{\rm{\ref{thm0.3}}}, and which
  is also smooth.
\end{cor}
\begin{proof}
  Given $\epsilon>0$, apply the construction of Theorem \ref{thm0.3}
  to ${\mathbf f}$ to get a continuous ${\mathbf g}_0$ satisfying
  (\ref{eq20}) with $\frac{\epsilon}2$, and a corresponding $\delta>0$.  By
  construction, for $\vec x\in\overline{B}^n(\vec p,\frac{\delta}2)$,
  $0<m\le\|{\mathbf g}_0(\vec x)\|<\frac{\epsilon}4$, and if $\mathbf
  f$ is smooth ($\cc^\infty$) on $\Omega$, then ${\mathbf g}_0$ is
  equal to a smooth map on the closed set $\Omega\setminus B^n(\vec
  p,\frac{\delta}2)$.  By the Whitney Approximation Theorem (\cite{lee}
  Ch.\ 6), there exists a smooth ${\mathbf g}:\Omega\to\re^q$ such
  that ${\mathbf g}\equiv{\mathbf g}_0\equiv{\mathbf f}$ on
  $\Omega\setminus B^n(\vec p,\frac{\delta}2)$, so (\ref{eq21}) is satisfied,
  and for all $\vec x\in\Omega$,
  $$\|{\mathbf g}(\vec x)-{\mathbf g}_0(\vec x)\|<\frac
  m2<\frac{\epsilon}8.$$ ${\mathbf g}$ is then close to $\mathbf f$
  (satisfying (\ref{eq20})):
  $$\|{\mathbf g}(\vec x)-{\mathbf f}(\vec x)\|\le\|{\mathbf g}(\vec
  x)-{\mathbf g}_0(\vec x)\|+\|{\mathbf g}_0(\vec x)-{\mathbf f}(\vec
  x)\|<\frac{\epsilon}8+\frac{\epsilon}2<\epsilon,$$ and has the
  claimed nonvanishing property (\ref{eq22}):
  \begin{equation}\label{eq16}
    \vec
  x\in B^n(\vec p,\frac{\delta}2)\implies\|{\mathbf g}(\vec x)\|\ge\|{\mathbf g}_0(\vec x)\|-\|{\mathbf
  g}(\vec x)-{\mathbf g}_0(\vec x)\|>m-\frac{m}2.
  \end{equation}
\end{proof}
\begin{cor}\label{cor0.5}
  Let $q$, ${\mathbf f}$, $\vec p$, and $R$ be as in Theorem
  \mbox{\rm{\ref{thm0.3}}}.  If $\vec p$ is an inessential zero of
  $\mathbf f$, and, additionally, $\mathbf f$ is a continuous,
  semialgebraic map, then there exists $\mathbf g$ as in Theorem
  \mbox{\rm{\ref{thm0.3}}}, and which is a continuous, semialgebraic
  map.
\end{cor}
\begin{proof}  
  Given $\epsilon>0$, apply the construction of Theorem \ref{thm0.3}
  to ${\mathbf f}$ to get a continuous ${\mathbf g}_0$ satisfying
  (\ref{eq20}) with $\frac{\epsilon}2$, and corresponding $\delta>0$, $m>0$.
  By construction, for $\vec x\in\overline{B}^n(\vec p,\frac{\delta}2)$,
  $\|{\mathbf f}(\vec x)\|<\frac{\epsilon}4$, and $m\le\|{\mathbf
  g}_0(\vec x)\|<\frac{\epsilon}4$.  Also, ${\mathbf f}-{\mathbf g}_0$
  is continuous on $\Omega$, and $\equiv{\mathbf0}$ on $S^{n-1}(\vec p,\frac{\delta}2)$, so by
  compactness, there is some $\delta_1$ so that
  $0<\delta_1<\frac{\delta}2$ and if $\delta_1\le\|\vec x-\vec
  p\|\le\frac{\delta}2$, then $\|{\mathbf f}(\vec x)-{\mathbf g}_0(\vec
  x)\|<\frac{m}2$.

  Apply the Weierstrass Approximation Theorem to ${\mathbf g}_0$ on
  the compact set $\overline{B}^n(\vec p,\frac{\delta}2)$ to get polynomials
  $h_1$, \ldots, $h_q$ so that for $\vec x\in\overline{B}^n(\vec
  p,\frac{\delta}2)$, the map ${\mathbf h}=(h_1,\ldots,h_q)$ satisfies
  \begin{equation}\label{eq32}
    \|{\mathbf h}(\vec x)-{\mathbf g}_0(\vec x)\|<\frac
  m2<\frac{\epsilon}8.
  \end{equation}
  Let $\chi:\re^1\to\re^1$ be a continuous,
  piecewise linear, weakly decreasing cutoff function:
  $$\chi(s)=\left\{\begin{array}{cl}1&s\le\delta_1\\\frac{s-\frac{\delta}2}{\delta_1-\frac{\delta}{2}}&\delta_1\le
  s\le\frac{\delta}2\\0&s\ge\frac{\delta}2\end{array}\right\}.$$
  Define ${\mathbf g}:\Omega\to\re^q$ by: 
  $${\mathbf g}(\vec x)=\chi(\|\vec x-\vec p\|)\cdot{\mathbf h}(\vec
    x)+(1-\chi(\|\vec x-\vec p\|))\cdot{\mathbf f}(\vec x).$$ By
    Proposition \ref{prop1.4}, $\mathbf g$ is semialgebraic, and by
    construction, $\mathbf g$ is continuous and satisfies (\ref{eq21}).

  ${\mathbf g}$ is close to $\mathbf f$ (satisfying (\ref{eq20})) ---
  either ${\mathbf g}(\vec x)={\mathbf f}(\vec x)$, or for $\vec x\in
  B^n(\vec p,\frac{\delta}2)$:
  \begin{eqnarray*}
    \|{\mathbf g}(\vec x)-{\mathbf f}(\vec x)\|&=&\|\chi\cdot({\mathbf
      h}(\vec x)-{\mathbf f}(\vec x))\|\\ &\le&\chi\cdot(\|{\mathbf
      h}(\vec x)-{\mathbf g}_0(\vec x)\|+\|{\mathbf g}_0(\vec
    x)-{\mathbf f}(\vec x))\|)\\ &<&\frac{\epsilon}8+\frac{\epsilon}2.
  \end{eqnarray*}
  $\mathbf g$ has the claimed nonvanishing property (\ref{eq22}); for
  $\vec x\in B^n(\vec p,\delta_1)$: $$\|{\mathbf g}(\vec x)\|=\|{\mathbf
  h}(\vec x)\|\ge\|{\mathbf g}_0(\vec x)\|-\|{\mathbf h}(\vec
  x)-{\mathbf g}_0(\vec x)\|>m-\frac m2,$$ and for $\vec x\in B^n(\vec
  p,\frac{\delta}2)\setminus B^n(\vec p,\delta_1)$:
  \begin{eqnarray*}
    \|{\mathbf g}(\vec x)\|&=&\|\chi\cdot{\mathbf h}(\vec
    x)+(1-\chi)\cdot{\mathbf f}(\vec x)\|\\ &=&\|{\mathbf g}_0(\vec
    x)+\chi\cdot({\mathbf h}(\vec x)-{\mathbf g}_0(\vec
    x))+(1-\chi)\cdot({\mathbf f}(\vec x)-{\mathbf g}_0(\vec
    x))\|\\ &\ge&\|{\mathbf g}_0(\vec x)\|-\chi\cdot\|{\mathbf h}(\vec
    x)-{\mathbf g}_0(\vec x)\|-(1-\chi)\cdot\|{\mathbf f}(\vec
    x)-{\mathbf g}_0(\vec x)\|\\ &>&m-\chi\cdot\frac
    m2-(1-\chi)\cdot\frac m2=\frac m2.
  \end{eqnarray*}
  The magnitude of ${\mathbf g}$ is also bounded above on
  $\overline{B}^n(\vec p,\frac{\delta}2)$:
  \begin{eqnarray}
    \|{\mathbf g}(\vec x)\|&=&\|\chi\cdot{\mathbf h}(\vec
    x)+(1-\chi)\cdot{\mathbf f}(\vec x)\|\nonumber\\
    &\le&\chi\cdot\bigl(\|{\mathbf h}(\vec x)-{\mathbf g}_0(\vec
    x)\|+\|{\mathbf g}_0(\vec x)\|\bigr)+(1-\chi)\cdot\|{\mathbf f}(\vec
    x)\|\nonumber\\
    &<&\chi\cdot\left(\frac{\epsilon}8+\frac{\epsilon}4\right)+(1-\chi)\cdot\frac{\epsilon}4<\frac{\epsilon}2.\label{eq23}
  \end{eqnarray}    
\end{proof}
\begin{rem}\label{rem2.4}
  The constructions of Corollaries \ref{cor0.4b} and \ref{cor0.5} are
  not compatible; it does not immediately follow from them that if
  ${\mathbf f}$ is both smooth and semialgebraic (for example,
  polynomial), then there exists ${\mathbf g}$ which is also both
  smooth and semialgebraic.  An analogue of Corollary \ref{cor0.4b}
  using Weierstrass Approximation could give a polynomial
  approximation to $\mathbf f$, at the cost of losing Property
  (\ref{eq21}) from Theorem \ref{thm0.3} and shrinking the radius $R$
  from (\ref{eq22}).
\end{rem}

\section{Homotopy through nonvanishing maps}\label{sec3}

The goal of this Section is to perturb a map ${\mathbf f}$ that has a
locally inessential zero by a homotopy $F(\vec x,t)$ which has no
nearby zeros for $t>0$.  This is a logically stronger property than
the negation of the following stability property: ``there exists some
$\epsilon>0$ so that for all $t\in[0,\epsilon]$, $F(\vec x,t)$ has a
zero near $\vec p$.''  In the $q=n$ case, such a homotopy through
nonvanishing maps is constructed by \cite{d84}, where ${\mathbf f}$
and $F$ are continuous (and further, invariant under a group action).
The novelty here is the generalization to $q\ne n$, and a construction
explicit enough to work in the semialgebraic category, using Corollary
\ref{cor0.5}.  The main step (\ref{eq2}) in the following Theorem (and
in that of \cite{d84} Proposition 1) is analogous to the well-known
``Alexander's Trick'' in topology (\cite{d}).

\begin{thm}\label{thm0.7}
  Let $q$, ${\mathbf f}$, $\vec p$, and $R$ be as in Theorem
  \mbox{\rm{\ref{thm0.3}}}.  If $\mathbf f$ has a locally
  inessential zero at $\vec p$, then for any $\rho$ with $0<\rho<R$, there
  exist a continuous map ${\mathbf j}:\Omega\to\re^q$ and a continuous
  homotopy $F:\Omega\times[0,1]\to\re^q$ from ${\mathbf f}$ to
  ${\mathbf j}$, such that:
  \begin{itemize}
    \item $F$ fixes the values of ${\mathbf f}$ outside an arbitrarily small ball:
  \begin{eqnarray}
    (\vec x,t)\in\left(\Omega\setminus B^n(\vec
     p,\rho)\right)\times[0,1]\implies F(\vec x,t)&=&{\mathbf f}(\vec
     x),\mbox{\ \ and}\label{eq7}
  \end{eqnarray}
    \item For every non-zero time $t$, $F(\vec x,t)$ is nonvanishing as a
  function of $\vec x$ near $\vec p$:
  \begin{eqnarray}
      t>0\implies V(F(\vec
     x,t))\cap B^n(\vec p,R)&=&\mbox{\rm{\O}}.\label{eq9}
  \end{eqnarray}
  \end{itemize}
  If, additionally, $\mathbf f$ is continuous and semialgebraic, then
  there exist $\mathbf j$ and $F$ as above which are also continuous
  and semialgebraic.
\end{thm}
\begin{proof}
  Pick any $\epsilon_1\in(0,1]$ and apply the construction of Theorem
    \ref{thm0.3} (or Corollary \ref{cor0.5} in the semialgebraic case)
    to ${\mathbf f}$ to get a continuous (respectively, semialgebraic)
    ${\mathbf g}:\Omega\to\re^q$ satisfying (\ref{eq20}) and a
    corresponding $\delta$ with $0<\delta<\min\{1,\rho\}$.  By
    construction, for $\vec x\in\overline{B}^n(\vec p,\frac{\delta}2)$,
    $\|{\mathbf f}(\vec x)\|<\frac{\epsilon_1}2$ and $\|{\mathbf g}(\vec
    x)\|<\frac{\epsilon_1}2$ (using (\ref{eq23}) in the semialgebraic case).
    There is also a lower bound on $\overline{B}^n(\vec p,\frac{\delta}2)$,
    $\|{\mathbf g}(\vec x)\|\ge\frac m2$, with $m$ from Theorem
    \ref{thm0.3}, and $\frac m2$ from Corollary \ref{cor0.5}.  After
    this point, we assume $\vec p=\vec0\in\Omega$ just to declutter
    the formulas.

  For each $t\in[0,\frac{\delta}2]$, define continuous, piecewise
  linear functions $\alpha_t:[0,\infty)\to(0,\infty)$ and
  $\beta_t:[0,\infty)\to(0,\infty)$, by:
  \begin{eqnarray}
    \alpha_t(s)&=&\left\{\begin{array}{cl}\frac{\delta}{t^2}&0\le s<\frac{t^2}2\\\left(\frac{1-\frac{\delta}{t^2}}{\frac{t^2}2}\right)(s-t^2)+1&\frac{t^2}2\le s< t^2\\1&t^2\le s\end{array}\right\},\label{eq3}\\
    \beta_t(s)&=&\left\{\begin{array}{cl}\frac{1-\frac{2t}{\delta}}{t}s+\frac{2t}{\delta}&0\le s< t\\ 1&t\le s\end{array}\right\}.\label{eq4}
  \end{eqnarray}
  So $\alpha_t$ is weakly decreasing in $s$, $\beta_t$ is weakly
  increasing, and in particular, for $t=0$,
  $\alpha_0\equiv\beta_0\equiv1$. 

  Based on formulas (\ref{eq3}) and (\ref{eq4}), define the following
  maps:
  \begin{eqnarray}
    \vec\alpha(\vec x,t)&=&\alpha_t(\|\vec x\|)\cdot\vec
    x\label{eq25}\\ \beta(\vec x,t)&=&\beta_{t}(\|\vec
    x\|).\label{eq26}
  \end{eqnarray}
  Both are defined for all $(\vec x,t)\in\re^n\times\re^1$, although
  neither $\vec\alpha$ nor $\beta$ is continuous.  Some elementary
  algebraic expansion of (\ref{eq3}) and (\ref{eq4}) with
  $s=\sqrt{x_1^2+\cdots+x_n^2}$ will show that the graph of
  $\vec\alpha$ is a semialgebraic set in
  $\re^n\times\re^1\times\re^n$, and the graph of $\beta$ is a
  semialgebraic set in $\re^n\times\re^1\times\re^1$.

  Define $F_0:\Omega\times[0,\frac{\delta}2]\to\re^q$ by:
  \begin{eqnarray}
    F_0(\vec x,t)&=&\left\{\begin{array}{cl}\beta_t(\|\vec
    x\|)\cdot{\mathbf g}(\alpha_t(\|\vec x\|)\cdot\vec x)&0\le\|\vec
    x\|<\frac{t^2}2\\\beta_t(\|\vec x\|)\cdot{\mathbf
    f}(\alpha_t(\|\vec x\|)\cdot\vec x)&\frac{t^2}2\le\|\vec
    x\|\end{array}\right\}\label{eq2}\\
    &=&\left\{\begin{array}{cl}\beta(\vec
    x,t)\cdot{\mathbf g}(\vec\alpha(\vec x,t))&0\le\|\vec
    x\|<\frac{t^2}2\\\beta(\vec x,t)\cdot{\mathbf f}(\vec\alpha(\vec
    x,t))&\frac{t^2}2\le\|\vec x\|\end{array}\right\}.\label{eq24}
  \end{eqnarray}
  By construction, $F_0(\vec x,0)={\mathbf f}(\vec x)$ for all
  $x\in\Omega$, and $F_0(\vec x,t)={\mathbf f}(\vec x)$ for all $\vec
  x\in\Omega$ with $\|\vec x\|>\frac{\delta}2$, so $F_0$ satisfies
  (\ref{eq7}).  If $\mathbf f$ and $\mathbf g$ are semialgebraic, then
  the expression (\ref{eq24}) shows that $F_0$ is semialgebraic by
  Proposition \ref{prop1.4}, although it will be easier to work with
  expression (\ref{eq2}).

  If $\|\vec x\|<\frac{t^2}2$, then $\|\alpha_t(\|\vec x\|)\cdot \vec
  x\|<\frac{\delta}{t^2}\cdot\frac{t^2}2=\frac{\delta}2<\delta$, so
  using the previously mentioned lower and upper bounds for
  $\|{\mathbf g}(\vec x)\|$, $$0<\frac{2t}{\delta}\cdot\frac
  m2<\|F_0(\vec x,t)\|=\beta_t(\|\vec x\|)\|{\mathbf
  g}(\alpha_t(\|\vec x\|)\cdot\vec x)\|<\frac{\epsilon_1}{2}.$$
  
   If $\frac{t^2}2\le\|\vec x\|<t^2$, then $\frac{t^2}2\le\|\vec
   x\|\le\|\alpha_t(\|\vec x\|)\cdot \vec x\|<\frac{\delta}{t^2}\cdot
   t^2=\delta$.  If $t^2\le\|\vec x\|<\delta$, then $\|\alpha_t(\|\vec
   x\|)\cdot \vec x\|=\|\vec x\|<\delta$, so
  \begin{eqnarray}
  0&\le&\frac{2t}{\delta}\min\{\|{\mathbf f}(\vec
   x)\|:\frac{t^2}2\le\|\vec x\|<\delta\}\label{eq27}\\
  &\le&\|F_0(\vec
   x,t)\|=\beta_t(\|\vec x\|)\|{\mathbf f}(\alpha_t(\|\vec
   x\|)\cdot\vec x)\|<\frac{\epsilon_1}{2},\label{eq28}
  \end{eqnarray}
  with equality in (\ref{eq27}) only at $t=0$.

  So, the nonvanishing property (\ref{eq9}) holds for $F_0$, together
  with an approximation property analogous to (\ref{eq20}):
  \begin{equation*}
    \|{\mathbf f}(\vec x)-F_0(\vec x,t)\|\le\|{\mathbf f}(\vec
    x)\|+\|F_0(\vec x,t)\|<\frac{\epsilon_1}{2}+\frac{\epsilon_1}{2}.
  \end{equation*}

  The two pieces of the formula (\ref{eq2}) agree on their common
  boundary where $\|\vec x\|=\frac{t^2}2>0$, $\alpha_t(\|\vec
  x\|)=\frac{\delta}{t^2}$, and $\|\alpha_t(\|\vec x\|)\cdot\vec
  x\|=\frac{\delta}2$, so $\alpha_t(\|\vec x\|)\cdot\vec x\in S^{n-1}(\vec
  p,\frac{\delta}2)$.  It follows from the continuity of (\ref{eq3})
  and (\ref{eq4}) in $s>0$ and $t>0$ that (\ref{eq25}), (\ref{eq26}),
  and $F_0$ are continuous for $\vec x\ne\vec0$ and $t>0$.  For a
  point $(\vec x,t)=(\vec0,t_0)$ with $t_0>0$, there is a neighborhood
  where
  $$F_0(\vec x,t)=\left(\frac{1-\frac{2t}{\delta}}{t}\|\vec
  x\|+\frac{2t}{\delta}\right)\cdot{\mathbf
  g}(\frac{\delta}{t^2}\cdot\vec x),$$ which is continuous.  For a
  point $(\vec x,t)=(\vec x_0,0)$ with $\vec x_0\ne\vec0$, there is a
  neighborhood where $$F_0(\vec x,t)=1\cdot{\mathbf f}(1\cdot\vec
  x),$$ which is also continuous.  It remains only to check continuity
  at $(\vec x,t)=(\vec0,0)$, where $F_0(\vec0,0)={\mathbf0}$.

  For any $\epsilon_2>0$, there is some $\delta_1>0$ so that if
  $\|\vec x\|<\delta_1$, then $\|{\mathbf f}(\vec x)\|<\epsilon_2$.
  Let $\delta_2=\min\{\delta,\delta_1\}$, and let
  $\delta_3=\frac{\epsilon_2}{1+\frac2{\delta}}>0$.  Continuity of
  $F_0$ at the origin will follow from showing that if $\|\vec
  x\|<\delta_2$ and $t<\delta_3$, then $\|F_0(\vec x,t)\|<\epsilon_2$.
  There are three cases.

  If $\|\vec x\|<\min\{\frac{t^2}2,\delta_2\}$, then 
  \begin{eqnarray*}
    \|F_0(\vec x,t)\|&=&\left(\frac{1-\frac{2t}{\delta}}{t}\|\vec
  x\|+\frac{2t}{\delta}\right)\|{\mathbf
  g}(\frac{\delta}{t^2}\cdot\vec x)\|\\
  &<&\left(\frac{(1-\frac{2t}{\delta})t}{2}+\frac{2t}{\delta}\right)\frac{\epsilon_1}{2}<\left(\frac{(1-\frac{2t}{\delta})}{2}+\frac{2}{\delta}\right)\delta_3\frac{1}{2}<\epsilon_2.
  \end{eqnarray*}
  If $\frac{t^2}2\le\|\vec x\|<\min\{t^2,\delta_2\}$, then 
  \begin{eqnarray*}
    \|F_0(\vec x,t)\|&=&\left(\frac{1-\frac{2t}{\delta}}{t}\|\vec
  x\|+\frac{2t}{\delta}\right)\|{\mathbf f}(\alpha_t(\|\vec
  x\|)\cdot\vec x)\|\\
  &<&\left((1-\frac{2t}{\delta})t+\frac{2t}{\delta}\right)\frac{\epsilon_1}{2}<\left((1-\frac{2t}{\delta})+\frac{2}{\delta}\right)\delta_3\frac{1}{2}<\epsilon_2.
  \end{eqnarray*}
  If $t^2\le\|\vec x\|<\delta_2$ then $$\|F_0(\vec
  x,t)\|=\beta_t(\|\vec x\|)\|{\mathbf f}(1\cdot\vec
  x)\|<1\epsilon_2$$ So $F_0$ is a continuous homotopy from ${\mathbf
    f}(\vec x)$ to ${\mathbf j}(\vec x)=F_0(\vec x,\frac{\delta}2)$
  for $0<t\le\frac{\delta}2$; re-scaling the $t$ variable to the
  interval $[0,1]$ gives $F$ and ${\mathbf j}$ satisfying
  \mbox{\rm{(\ref{eq7})}} and \mbox{\rm{(\ref{eq9})}} as claimed.
\end{proof}
\begin{rem}\label{rem3.3}
  It would be easy to construct another homotopy from ${\mathbf j}$ to
  a smooth map ${\mathbf j}_2$, by a continuous homotopy of
  nonvanishing approximations.  Concatenation would then give a
  continuous homotopy $F$ from the original map $\mathbf f$ to the
  smooth map ${\mathbf j}_2$.  If ${\mathbf f}$ were also smooth, then
  the existence of a continuous homotopy implies the existence of a
  smooth homotopy $\tilde F$ from $\mathbf f$ to ${\mathbf j}_2$
  (\cite{lee} Theorem 6.29), and Property (\ref{eq7}) could be
  arranged.  However, finding a smooth, or even $\cc^1$, homotopy also
  satisfying (\ref{eq9}) seems to be a difficult problem, which we
  state here as a Question.
\end{rem}
\begin{question}\label{question4.3}
  Let $q$, ${\mathbf f}$, $\vec p$, and $R$ be as in Theorem
  \mbox{\rm{\ref{thm0.3}}}, where $\mathbf f$ has a locally
  inessential zero at $\vec p$.  Suppose ${\mathbf f}$ is [$\cc^1$,
  smooth, real analytic, polynomial] on $\Omega$.  Does there always
  exist a [$\cc^1$, smooth, real analytic, polynomial] map $F:B^n(\vec
  p,\rho)\times(-1,1)\to\re^q$ for some $0<\rho<R$ with the following
  properties?
  \begin{itemize}
    \item $F$ agrees with ${\mathbf f}$ for all $\vec x\in B^n(\vec
      p,\rho)$ at time $t=0$: $F(\vec x,0)={\mathbf f}(\vec x)$, and
  \item for every non-zero time $t$, $F(\vec x,t)$ is nonvanishing as
    a function of $\vec x$: $t\ne0\implies V(F(\vec
    x,t))=\mbox{\rm{\O}}$.
  \end{itemize}
\end{question}
The Question can be considered as a classic extension problem: does a
$\re^q$ valued ${\mathbf f}$ with a locally inessential zero in
$\re^n$ extend, locally, to a $\re^q$ valued function $F$ with an
isolated zero in $\re^{n+1}$, with the same regularity?  Theorem
\ref{thm0.7} constructs such a continuous one-sided extension for
$0\le t\le 1$; then $F(\vec x,t^2)$ for $-1\le t\le1$ is a two-sided
extension.  This also solves the two-sided extension problem in the
continuous, semialgebraic case.

We are not conjecturing an answer either way to any of the stated
versions of Question \ref{question4.3}; it would be interesting to get
either a proof of ``yes'' or a concrete counterexample for ``no'' for
any of the cases.  It is possible that the answer will depend on the
dimensions $n$ and $q$, in analogy with exotic smoothness phenomena.
\begin{cor}\label{cor5.4}
  Given $n$ and $q$, if the answer to Question
  {\rm{\ref{question4.3}}} is ``yes'' in the case where ``real
  analytic'' is stated in both the hypothesis and the conclusion, then
  the answer is also yes in the polynomial case.
\end{cor}
\begin{proof}
  Assume as before that $\vec p=\vec0$.  Suppose that ${\mathbf
    f}(\vec{x})$ as in Question \ref{question4.3} is a polynomial map
  with an isolated zero at $\vec 0$, which extends to a real analytic
  $F(\vec{x},t)$ on some $B^n(\vec0,\rho)\times(-1,1)$ with an
  isolated zero at $(\vec 0,0)$.  By the \L ojasiewicz inequality,
  there are some positive constants $C_1$ and $\eta$ so that in some
  possibly smaller neighborhood of $(\vec0,0)$, $$||F(\vec x,t)||\ge
  C_1\|(\vec x,t)\|^\eta.$$ Choose an integer
  $D\ge\max\{\eta,\mbox{degree}({\mathbf f})\}$; then $F$ has a degree
  $D$ Taylor polynomial $P(\vec x,t)$ that satisfies $P(\vec
  x,0)={\mathbf f}(\vec x)$, and there is some constant $C_2>0$ so
  that in some neighborhood of the origin, $\|F(\vec x,t)-P(\vec
  x,t)\|\le C_2\|(\vec x,t)\|^{D+1}$.  So, $P$ also has a unique zero
  in a sufficiently small neighborhood of $(\vec0,0)$:
\begin{eqnarray*}
  \|P(\vec x,t)\|&\ge&||F(\vec x,t)||-\|F(\vec x,t)-P(\vec
  x,t)\|\\ &\ge&C_1\|(\vec x,t)\|^\eta-C_2\|(\vec
  x,t)\|^{D+1}\\ &=&C_1\|(\vec
  x,t)\|^\eta\left(1-\frac{C_2}{C_1}\|(\vec
  x,t)\|^{D+1-\eta}\right).
\end{eqnarray*}
The $t$ variable can be re-scaled as needed.
\end{proof}

The converse of Theorem \ref{thm0.7} also holds (in the continuous
case), so this gives a sixth equivalent property for Lemma
\ref{lem3.2}.
\begin{thm}\label{4.5}
  For a continuous map ${\mathbf f}:\Omega\to\re^q$, and an isolated
  zero $\vec p\in B^n(\vec p,R)$ as in Definition \ref{def3.1}, the
  following are equivalent:
  \begin{itemize}
    \item $\vec p$ is a locally inessential zero of $\mathbf f$;
    \item There exist some $r_0\in(0,R)$ and a continuous $F:B^n(\vec
      p,r_0)\times[0,1]\to\re^q$ such that $F(\vec x,0)={\mathbf
        f}(\vec x)$ for $\vec x\in B^n(\vec p,r_0)$ and $t>0\implies
      F(\vec x,t)\ne{\mathbf0}$.
  \end{itemize}
\end{thm}
\begin{proof}
  The existence of such a homotopy $F$ implies Property (iv) from
  Lemma \ref{lem3.2}, as follows.  For any $\epsilon>0$, there is some
  $\delta$ with $0<\delta\le\min\{1,r_0\}$ and so that if $\|\vec
  x-\vec p\|<\delta$ and $0\le t<\delta$, then $\|F(\vec
  x,t)\|<\epsilon$.  For any $r\in(0,\delta)$, define
  $\Theta_r:B^n_*(\vec p,r)\times[0,1]\to\re^q\setminus\{{\mathbf0}\}$
  by $\Theta_r(\vec x,t)=F(\vec p+(1-t)\cdot(\vec x-\vec
  p),\frac{\delta}2t)$.  Then, $\Theta_r(\vec x,0)=F(\vec
  x,0)={\mathbf f}(\vec x)$, $\Theta_r(\vec x,1)\equiv F(\vec
  p,\frac{\delta}2)$, and $\Theta_r(\vec x,t)\in
  B^q_*({\mathbf0},\epsilon)$ by construction.
\end{proof}
  
\section{Real analytic maps to the plane}\label{sec4}

The following Theorem gives an answer to a version of Question
\ref{question4.3} in the special case where the target dimension is
$q=2$ and the given data ${\mathbf f}$ is real analytic on
$\Omega\subseteq\re^n$, for any $n\ge1$.  The construction of a
H\"older continuous extension $F$ which is real analytic except at the
origin, where it satisfies a one-point Lipschitz condition $\|F(\vec
x,t)\|\le C\|(\vec x-\vec p,t)\|$, does not use the Theorems of
Sections \ref{sec2} or \ref{sec3}.  One step refers to Lemma
\ref{lem5.3}, which appears after the main Proof.
\begin{thm}\label{thm5.1}
  Let ${\mathbf f}:\Omega\to\re^2$ be a real analytic map so that
  $\mathbf f$ has an isolated zero $\vec p\in\Omega$.  If either:
  \begin{itemize}
    \item $n\ne2$; or,
    \item $n=2$ and $\vec p$ is a locally inessential zero of $\mathbf
      f$,
  \end{itemize}
  then there exist some $\rho>0$, and a H\"older continuous map
  $F:B^{n+1}((\vec p,0),\rho)\to\re^2$ such that:
  \begin{itemize}
  \item  $F$ locally extends $\mathbf f$: $F(\vec x,0)={\mathbf f}(\vec x)$
    for $\vec x\in B^{n}(\vec p,\rho)$, and
  \item on the punctured ball $B^{n+1}_*((\vec p,0),\rho)$, $F$ is real
   analytic and nonvanishing, and
  \item there is some $C>0$ so that for all $(\vec x,t)\in
    B^{n+1}_*((\vec p,0),\rho)$,
    \begin{equation}\label{eq31}
      \|F(\vec x,t)\|\le C\|(\vec x-\vec p,t)\|.
    \end{equation}
  \end{itemize}
\end{thm}
\begin{proof}
  Let $\vec p=\vec0$ and let $R>0$, as in Theorem \ref{thm0.3}.
  Consider $\mathbf f$ as a complex valued function ${\mathbf f}(\vec
  x)=(f_1(\vec x),f_2(\vec x))=f_1+if_2$.

  Case 1.  We first consider the case $n=1$, where there is a simple
  proof and a stronger result, neither of which we have been able to
  generalize to higher $n$ or $q$.  By the real analytic assumption,
  there is some $\rho_0>0$ so that for $|x_1|<\rho_0$,
  $g(x_1)=f_1(x_1)+if_2(x_1)$ is equal to a non-constant convergent
  series with complex coefficients
  ${\displaystyle{\sum_{k=1}^\infty(a_k+ib_k)x_1^k}}$.  Replacing the
  real variable $x_1$ with a complex variable $z=x_1+it$ gives, for
  the same radius $|z|<\rho_0$, a series converging to a holomorphic
  function $F(z)$ on $B^2((0,0),\rho_0)\subseteq\co$ which extends
  $g=f_1+if_2$ and whose zeros are all isolated.

  Case 2.  $n\ge2$.  The first three steps in the proof are
  preparation steps for $\mathbf f$ in a small neighborhood; Step 4
  constructs a $k^{th}$ root, and the extension is constructed in the
  remaining steps.  The locally inessential property is used only in
  Step 1, although as mentioned in Remark \ref{rem3.6}, any isolated
  zero is locally inessential in the $n\ne2$ case, by Lemma
  \ref{lem3.5}.

  Step 1.  (Normalization on a boundary sphere) Using Property (iv)
  from Lemma \ref{lem3.2}, there exists some
  $0<\rho_1<\min\{R,\frac12\}$ so that ${\mathbf f}$ restricted to
  $B^n_*(\vec0,\rho_1)\to\co\setminus\{0+0i\}$ is homotopic in
  $\co\setminus\{0+0i\}$ to the constant map with image $\{1+0i\}$.
  By the usual construction of the universal covering space
  $\exp:\co\to\co\setminus\{0+0i\}$, there is a branch of the
  logarithm and a lift of $\mathbf f$ (\cite{s} Theorem 2.4.5) to a
  real analytic composite $\log\circ{\mathbf
    f}:B^n_*(\vec0,\rho_1)\to\co$ so that $\exp(\log({\mathbf f}(\vec
  x)))={\mathbf f}(\vec x)$.  For any $0<\rho_2<\rho_1$, the
  restriction of $\log\circ{\mathbf f}$ to the sphere
  $S^{n-1}(\vec0,\rho_2)$ is real analytic, and solving the classical
  Dirichlet problem gives a function ${\mathbf u}(\vec x)$ which is
  harmonic and real analytic on some ball $B^n(\vec0,\rho_3)$,
  $\rho_2<\rho_3<\rho_1$, and such that ${\mathbf u}(\vec
  x)=\log({\mathbf f}(\vec x))$ for $\vec x\in S^{n-1}(\vec0,\rho_2)$.
  Let ${\mathbf m}(\vec x)=\exp(-{\mathbf u}(\vec x))$, so by
  construction, the complex product ${\mathbf m}\cdot{\mathbf f}$ is
  real analytic on $B^n(\vec0,\rho_3)$, ${\mathbf m}(\vec x){\mathbf
    f}(\vec x)\equiv1+0i$ for all $\vec x\in S^{n-1}(\vec0,\rho_2)$,
  $({\mathbf m}\cdot{\mathbf f})^{-1}(\{0+0i\})=\{\vec0\}$, and for $n=2$,
  ${\mathbf m}\cdot{\mathbf f}$ still has a locally inessential zero.

  Step 2.  (Polar coordinates) Using the exponential covering space
  again with base point 
  \begin{equation}\label{eq29}
    ({\mathbf m}\cdot{\mathbf
  f})(\rho_2,0,\ldots,0)=1+0i,
  \end{equation}
  there is a branch of the logarithm on
  $B^n_*(\vec0,\rho_3)$ with real and imaginary parts:
  $$\log(({\mathbf m}\cdot{\mathbf f})(\vec x))=\ln|({\mathbf
  m}\cdot{\mathbf f})(\vec x)|+i\theta(\vec x),$$ where
  $\theta(\rho_2,0,\ldots,0)=0$.  Exponentiating, $$({\mathbf
  m}\cdot{\mathbf f})(\vec x)=r(\vec x)\cdot\exp(i\theta(\vec x)),$$
  for $\theta$ and $r>0$ real analytic on
  $B^n_*(\vec0,\rho_3)$.

  Step 3.  (Boundedness of $\theta$) Using (\ref{eq29}), the real
  analytic function $\rp({\mathbf m}\cdot{\mathbf f})$ is not
  identically zero on any open interval of the $x_1$-axis
  $\{(x_1,0,\ldots,0)\}\cap B^n(\vec0,\rho_3)$.  So, by the Weierstrass
  Preparation Theorem, there exist some radius $0<\rho_4<\rho_2$, some
  degree $N$, and some real analytic functions $\nu$, $c_1$, \ldots,
  $c_N$, so that for all $\vec x\in B^n(\vec0,\rho_4)$, 
  $$\rp({\mathbf m}(\vec x){\mathbf f}(\vec x))=\nu(\vec
  x)\cdot(x_1^N+c_1(x_2,\ldots,x_n)x_1^{N-1}+\ldots+c_N(x_2,\ldots,x_n)),$$
  where $\nu(\vec x)$ is nonvanishing.  Denote the open cylinder
  $$\Gamma=\left(-\frac{\rho_4}2,\frac{\rho_4}2\right)\times
  B^{n-1}(\vec0,\frac{\rho_4}2)\subseteq B^n(\vec0,\rho_4).$$ On the
  sphere $S^{n-1}(\vec0,\rho_2)$, $\theta\equiv0$, and on the complement
  ${\overline B}^n(\vec0,\rho_2)\setminus \Gamma$, $\theta$ is bounded by
  compactness: $|\theta|\le K$.  For a fixed
  $x^\prime=(x_2^\prime,\ldots,x_n^\prime)\in
  B^{n-1}(\vec0,\frac{\rho_4}2)$, consider the following expression as
  a function of $x_1$ only:
  \begin{eqnarray*}
  \rp({\mathbf m}(x_1,x^\prime){\mathbf
    f}(x_1,x^\prime))&=&\nu(x_1,x^\prime)\left(x_1^N+c_1(x^\prime)x_1^{N-1}+\ldots+c_N(x^\prime)\right)\\ &=&|{\mathbf
    m}(x_1,x^\prime){\mathbf
    f}(x_1,x^\prime)|\cos(\theta(x_1,x^\prime)) \ \ \ \ \mbox{(for
    $(x_1,x^\prime)\ne\vec0$)}.
  \end{eqnarray*}
  By the Fundamental Theorem of Algebra, for a fixed $x^\prime\in
  B^{n-1}_*((0,\ldots,0),\frac{\rho_4}2)$,
  $\cos(\theta(x_1,x^\prime))$ can have at most $N$ zeros on the
  interval $-\frac{\rho_4}2<x_1<\frac{\rho_4}2$, so
  $|\theta(x_1,x^\prime)|$ is bounded by $K+2\pi N$ for all $x_1$ with
  $(x_1,x^\prime)\in{\overline B}^n(\vec0,\rho_2)$, and this bound does
  not depend on $x^\prime$.  For points on the $x_1$ axis, where
  $x_1\ne0$ and $x^\prime=(0,\ldots,0)$, $|\theta|$ has the same bound
  by continuity.  By continuity at points on the boundary, there is
  some radius $\rho_5$, $\rho_2<\rho_5<\rho_3$, so that $\theta$ is
  bounded on ${\overline B}^n_*(\vec0,\rho_5)$.

  Step 4.  (Taking a root) By the boundedness of $\theta$, there is
  some integer $k$ so that $\frac{|\theta(\vec x)|}{k}<\frac{\pi}2$ for $\vec
  x\in{\overline B}^n_*(\vec0,\rho_5)$.  The following
  $k^{th}$ root is well-defined and continuous on ${\overline
  B}^n(\vec0,\rho_5)$:
  $$({\mathbf m}(\vec x){\mathbf f}(\vec
  x))^{1/k}=\left\{\begin{array}{cl}(r(\vec x))^{1/k}\exp(i\theta(\vec
  x)/k)&\vec x\ne\vec0\\0&\vec x=\vec0\end{array}\right\},$$ and
  satisfies 
  \begin{equation}\label{eq30}
    \rp\left(({\mathbf m}(\vec x){\mathbf f}(\vec x))^{1/k}\right)\ge0.
  \end{equation}
  In the open set $B^n_*(\vec0,\rho_5)$, $({\mathbf m}(\vec x){\mathbf f}(\vec
  x))^{1/k}$ is real analytic and nonvanishing (with positive real
  part), and is $\equiv1+0i$ on $S^{n-1}(\vec0,\rho_2)$.

  Because ${\mathbf m}\cdot{\mathbf f}$ is real analytic on
  $B^n(\vec0,\rho_3)$, its components have bounded gradient on the
  closed ball ${\overline B}^n(\vec0,\rho_5)$ and ${\mathbf
  m}\cdot{\mathbf f}$ satisfies a uniform Lipschitz condition on
  $B^n(\vec0,\rho_5)$.  Lemma \ref{lem5.3} applies, to show that
  $({\mathbf m}(\vec x){\mathbf f}(\vec x))^{1/k}$ is H\"older
  continuous on $B^n(\vec0,\rho_5)$ with exponent $\frac1k$.

  Step 5.  (Solving a boundary value problem) Now consider $\re^{n+1}$
  with coordinates $(\vec x,t)$, and the closed ball ${\overline
  B}^n(\vec0,\rho_2)\times\{0\}$ as the equatorial disk of the closed
  ball ${\overline B}^{n+1}((\vec0,0),\rho_2)$.  Let
  $S^n_+((\vec0,0),\rho_2)$ denote the upper hemisphere as in Notation
  \ref{not0.1}, so that the boundary of the upper half ball
  $B^{n+1}_+((\vec0,0),\rho_2)$ is the union
  $S^n_+((\vec0,0),\rho_2)\cup({\overline
  B}^n(\vec0,\rho_2)\times\{0\})$.  On this boundary set, the
  following function is continuous:
  $${\mathbf h}(\vec x,t)=\left\{\begin{array}{cl}({\mathbf m}(\vec
  x){\mathbf f}(\vec x))^{1/k}&\vec x\in {\overline
    B}^n(\vec0,\rho_2),\ t=0\\1+0i&(\vec x,t)\in
  S^n_+((\vec0,0),\rho_2)\end{array}\right\}.$$ Solving the classical
  Dirichlet problem extends $\mathbf h$ to a complex valued continuous
  function $H(\vec x,t)$ on the closure of
  $B^{n+1}_+((\vec0,0),\rho_2)$ such that $H(\vec x,t)$ is harmonic on
  $B^{n+1}_+((\vec0,0),\rho_2)$.  By the maximum principle applied to
  the harmonic real function $-\rp(H(\vec x,t))$ and (\ref{eq30}),
  $\rp(H(\vec x,t))$ is strictly positive on the interior and attains
  its minimum value $0$ only at the origin.  Near boundary points
  $(\vec x,0)\in B^n_*(\vec0,\rho_2)\times\{0\}$, $H$ extends uniquely
  and real analytically across the boundary into the lower half space
  $\{t<0\}$.  For any $\rho_6$ with $0<\rho_6<\rho_2$, the restriction
  of $H(\vec x,t)$ to the closure of the smaller half ball
  $B^{n+1}_+((\vec0,0),\rho_6)$ is H\"older continuous with the same
  exponent, $\alpha=\frac1k$, as the data on the flat part of the
  boundary.

  Step 6.  (Constructing the extension) The composite $H_2(\vec
  x,t)=H(\vec x,t^2)$ is defined and H\"older continuous on the whole
  ball $B^{n+1}((\vec0,0),\rho_6)$: using $\rho_6<\frac12$,
  \begin{eqnarray*}
    |H_2(\vec y,s)-H_2(\vec x,t)|&=&|H(\vec y,s^2)-H(\vec
    x,t^2)|\\ &\le&C_3\|(\vec y,s^2)-(\vec x,t^2)\|^{1/k}\le
    C_3\left[\|\vec x-\vec
      y\|^2+|s^2-t^2|^2\right]^{1/(2k)}\\ &\le&C_3\left[\|\vec x-\vec
      y\|^2+|s-t|^2\right]^{1/(2k)}\le C_3\|(\vec y,s)-(\vec
    x,t)\|^{1/k}.
  \end{eqnarray*}
  $H(\vec x,t^2)$ is real analytic except at the origin, at which the
  H\"older condition gives, for all $(\vec x,t)$ in
  $B^{n+1}((\vec0,0),\rho_6)$:
  \begin{equation}\label{eq35}
  |H_2(\vec x,t)|=|H(\vec x,t^2)|\le C_3\|(\vec x,t^2)\|^{1/k}\le
  C_3\|(\vec x,t)\|^{1/k}.
  \end{equation}
  The $k^{th}$ power $(H(\vec x,t^2))^k$ is similarly H\"older
  continuous with the same exponent $\frac1k$ on the same ball, and
  real analytic except at the origin, and from (\ref{eq35}), it
  satisfies:
  $$\left|(H(\vec x,t^2))^k\right|\le C_3^k\|(\vec x,t)\|.$$ By
  construction, $(H(\vec x,t^2))^k$ has a unique zero at $(\vec0,0)$,
  and for $t=0$, $(H(\vec x,0))^k={\mathbf m}(\vec x){\mathbf f}(\vec
  x)$.  An extension $F$ of ${\mathbf f}$ as claimed, with
  $\rho=\rho_6$ and satisfying (\ref{eq31}), is: $$F(\vec
  x,t)=\frac{(H(\vec x,t^2))^k}{{\mathbf m}(\vec x)}.$$
\end{proof}
Such an extension $F$ may vanish to higher order but would still not
necessarily have a continuous derivative at the origin; the above
argument also does not show that $F$ is (uniformly) Lipschitz
continuous in a neighborhood of the origin.
\begin{rem}\label{rem5.2}
  The Proof of Theorem \ref{thm5.1} used some facts about harmonic
  functions that are well-known to PDE experts.  The existence of a
  solution of the Dirichlet problem, to construct $\mathbf u$ in Step
  1, and to construct $H$ in Step 5, is given by Theorem 2.14 of
  \cite{gt}, using only that the boundary data is continuous and the
  domain has a sufficiently regular boundary.  The maximum principle
  (which was the key step for the nonvanishing) holds for any bounded
  domain.  The fact that $\mathbf u$ and $H$ extend real analytically
  across the boundary in neighborhoods where the boundary and the
  Dirichlet data are real analytic follows from a standard argument
  using the Cauchy-Kovalevskaya Theorem and the reflection principle
  for harmonic functions (\cite{g}).  The $\cc^{0,\alpha}$ H\"older
  property for the harmonic function $H(\vec x,t)$ up to a part of the
  boundary where the boundary values are H\"older continuous, from
  Step 5, is the deepest result used in the Proof; it also depends, in
  general, on the geometry of the boundary (\cite{aikawa}, \cite{lu},
  \cite{m}).
\end{rem}

\begin{lem}\label{lem5.3}
  Given $R>0$, integers $n\ge1$, $k\ge1$, and a continuous function
  ${\mathbf g}:B^n(\vec0,R)\to\co$ such that
  ${\mbox{\rm{$\rp$}}}({\mathbf g}(\vec x))\ge0$, with equality only
  at ${\mathbf g}(\vec 0)=0+0i$, if the {\mbox{$k^{th}$}} power
  $({\mathbf g}(\vec x))^k$ is Lipschitz continuous on $B^n(\vec0,R)$:
  for some $C_1$ and any $\vec x,\vec y\in B^n(\vec0,R)$,
  $$\left|({\mathbf g}(\vec y))^k-({\mathbf g}(\vec x))^k\right|\le
  C_1\|\vec y-\vec x\|,$$ then $\mathbf g$ is H\"older continuous on
  $B^n(\vec0,R)$: for some $C_2$ and any $\vec x,\vec y\in
  B^n(\vec0,R)$, 
  \begin{equation}\label{eq34}
    \left|{\mathbf g}(\vec y)-{\mathbf g}(\vec
  x)\right|\le C_2\|\vec y-\vec x\|^{1/k}.
  \end{equation}
\end{lem}
\begin{proof}
  For $z\in\co\setminus\{0+0i\}$, denote by ${\rm{Arg}}(z)$ the angle
  $\vartheta\in(-\pi,\pi]$ so that $z=|z|e^{i\vartheta}$.

  Step 1.  Claim: For $z,w\in\co\setminus\{0+0i\}$, if
  $|{\rm{Arg}}(\frac zw)|<\frac{\pi}{2k}$ then $|w-z|^k\le|w^k-z^k|$,
  with equality only if $k=1$ or $z=w$.  The Claim is trivial for
  $k=1$; for $k>1$, the proof of the Claim has two cases.

  Case 1. $|z|\le|w|$.  Let $\zeta=\frac zw$.  By the hypothesis
  $|{\rm{Arg}}(\zeta)|<\frac{\pi}{2k}$, for $j=0,1,\ldots,k-1$,
  $\rp(\zeta^j)>0$, so $\rp(1+\zeta+\cdots+\zeta^{k-1})>1$ and
  $\left|1+\zeta+\cdots+\zeta^{k-1}\right|>1$.
  $|{\rm{Arg}}(\zeta)|<\frac{\pi}{2k}$ and $|\zeta|<1$ also imply
  $|1-\zeta|<1$, so
  \begin{equation}\label{eq33}
    |1-\zeta|^k\le|1-\zeta|\le|1-\zeta|\left|1+\zeta+\cdots+\zeta^{k-1}\right|=|1-\zeta^k|,
  \end{equation}
  with equality only if $\zeta=1$.  Multiplying both sides by $|w|^k$
  establishes the Claim.

  Case 2.  $|w|<|z|$.  Let $\zeta=\frac wz$, so
  $|{\rm{Arg}}(\zeta)|<\frac{\pi}{2k}$ and $|\zeta|<1$ still hold, so
  (\ref{eq33}) follows, and then multiplying by $|z|^k$ establishes
  the Claim.

  Step 2.  The property (\ref{eq34}) clearly holds for $\vec x=\vec y$
  or $k=1$, so the following cases will assume $k>1$ and $\vec
  x\ne\vec y$.

  Case 1.  For $\vec x=\vec0$, where ${\mathbf g}(\vec x)=0+0i$, the
  conclusion $|{\mathbf g}(\vec y)|\le C_2\|\vec y\|^{1/k}$ follows
  from the assumption $|({\mathbf g}(\vec y))^k-(0+0i)^k|\le C_1\|\vec
  y-\vec0\|$, with $C_2=C_1^{1/k}$.  The case $\vec y=\vec0$ is
  analogous.

  Case 2.  If $\vec y=\lambda\vec x$ for some
  $\lambda<0$, so that $\vec0$ is between $\vec x$ and $\vec y$ in
  $B^n(\vec0,R)$, then $\|\vec x\|<\|\vec y-\vec x\|$ and $\|\vec
  y\|<\|\vec y-\vec x\|$.  Then using Case 1., $$|{\mathbf g}(\vec
  y)-{\mathbf g}(\vec x)|\le|{\mathbf g}(\vec y)|+|{\mathbf g}(\vec
  x)|\le C_1^{1/k}\|\vec y\|^{1/k}+C_1^{1/k}\|\vec
  x\|^{1/k}<2C_1^{1/k}\|\vec y-\vec x\|^{1/k}.$$

  For the remaining cases, with $\vec x,\vec y\in B_*^n(\vec0,R)$, let
  $z={\mathbf g}(\vec x)$ and $w={\mathbf g}(\vec y)$, so $\rp(z)>0$
  and $\rp(w)>0$ by hypothesis, and ${\rm{Arg}}(\frac
  zw)\in(-\pi,\pi)$.

  Case 3.  If $|{\rm{Arg}}(\frac zw)|<\frac{\pi}{2k}$, then the Claim
  from Step 1.\ applies and $$\left|{\mathbf g}(\vec y)-{\mathbf
  g}(\vec x)\right|^k\le\left|({\mathbf g}(\vec y))^k-({\mathbf
  g}(\vec x))^k\right|\le C_1\|\vec y-\vec x\|,$$ which gives
  (\ref{eq34}) with $C_2=C_1^{1/k}$.

  Case 4.  Suppose $|{\rm{Arg}}(\frac zw)|\ge\frac{\pi}{2k}$ and the
  line segment connecting $\vec x$ to $\vec y$ does not meet the
  origin: $\sigma:[0,1]\to B^n_*(\vec0,R)$, $\sigma(0)=\vec x$,
  $\sigma(1)=\vec y$.  Then ${\rm{Arg}}\circ{\mathbf
    g}\circ\sigma:[0,1]\to(-\frac{\pi}2,\frac{\pi}2)$ is well-defined
  and continuous, and by the Intermediate Value Theorem, there exist
  $2k+1$ points $t_0=0<t_1<t_2<\ldots<t_{2k}=1$ so that
  $\left|{\rm{Arg}}({\mathbf g}(\sigma(t_{j})))-{\rm{Arg}}({\mathbf
    g}(\sigma(t_{j-1})))\right|<\frac{\pi}{2k}$, so by Case 3.,
  $$\left|{\mathbf g}(\sigma(t_{j}))-{\mathbf
  g}(\sigma(t_{j-1}))\right|\le C_1^{1/k}\|\sigma(t_{j})-\sigma(t_{j-1})\|^{1/k}< C_1^{1/k}\|\vec
  y-\vec x\|^{1/k}.$$ Then (\ref{eq34}) follows, with
  $C_2=2kC_1^{1/k}$:
  $$|{\mathbf g}(\vec y)-{\mathbf g}(\vec
  x)|=\left|\sum_{j=1}^{2k}{\mathbf g}(\sigma(t_{j}))-{\mathbf
    g}(\sigma(t_{j-1}))\right|\le\sum_{j=1}^{2k}\left|{\mathbf
    g}(\sigma(t_{j}))-{\mathbf g}(\sigma(t_{j-1}))\right|<
  2kC_1^{1/k}\|\vec y-\vec x\|^{1/k}.$$

  The above four Cases show that for $k>1$, (\ref{eq34}) is satisfied
  for any $\vec x$, $\vec y$ by choosing the maximum constant
  $C_2=2kC_1^{1/k}$.
\end{proof}

\section{Polynomial examples}\label{sec5}
As mentioned in the Introduction, the $n=3$, $q=2$ case of an isolated
zero of $\mathbf f$ can be viewed as an isolated point $\vec p$ in the
intersection of two real surfaces $V(f_1)\cap V(f_2)$ in $\re^3$,
which generically would meet in a space curve.  In applications of
geometry, it may be of interest to define a
space curve implicitly by two polynomials, and then to remove any
isolated points, which can be done by a continuous, semialgebraic
homotopy as in Theorem \ref{thm0.7}, or by a homotopy as in Theorem
\ref{thm5.1} which is real analytic for $(\vec x,t)$ near but not
equal to $(\vec p,0)$.
\begin{example}\label{ex0.5}
  For $n=3$, $q=2$, where $\pi_2(S^1)\cong\{0\}$, consider the
  following pair of polynomials in $\re^3$, so that the varieties $V(f_1)$
  and $V(f_2)$ meet only at $\vec0$:
  \begin{align*}
    f_1(x,y,z) &= 8x^2+8y^2-z^2 & &  \text{Cone  }\\
    f_2(x,y,z) &= z(x^2+y^2)-x^3 & & \text{Cartan Umbrella}
  \end{align*}
  Corollary \ref{cor0.5}, applied to ${\mathbf f}=(f_1,f_2)$, shows
  that there exists some semialgebraic $\mathbf g$ close to ${\mathbf
  f}$ so that $V({\mathbf g})=\mbox{\O}$.  However, in this case, it
  is not possible to construct $\mathbf g$ by merely translating
  the varieties:
  $$V(f_1(\vec x-\vec\tau_1), f_2(\vec x-\vec\tau_2))\ne\mbox{\O},$$
  nor by choosing other level sets:
  $$V(f_1+C_1,f_2+C_2)\ne\mbox{\O}$$ for any constants $C_1$, $C_2$.
  Theorem \ref{thm0.7} also applies, to show that there exists a
  continuous, semialgebraic homotopy that removes the isolated point
  of intersection, and analogously, Theorem \ref{thm5.1} also applies.
  For this example, there is such a homotopy given globally by
  polynomials: $F=(F_1,F_2):\re^3\times\re^1\to\re^2$:
   \begin{eqnarray*}
    F_1(x,y,z,t)&=&8x^2+8y^2-z^2+t^2\\
    F_2(x,y,z,t)&=&z(x^2+y^2)+zt^2-x^3.
  \end{eqnarray*}
\end{example}
\begin{example}\label{ex0.6}
  For $n=4$, $q=3$, $\pi_3(S^2)\not\cong\{0\}$.  A map not homotopic
  to a constant is given by the restriction of this polynomial map
  $\co^2=\re^4\to\re^3$:
  $${\mathbf f}(z_1,z_2)=(2z_1\z_2,|z_1|^2-|z_2|^2)$$ to the unit
  sphere $S^3(\vec0,1)=\{|z_1|^2+|z_2|^2=1\}$.  The restriction is the
  Hopf map (\cite{ac} Example 4.6; \cite{ac2} \S8.2), and this map
  $\mathbf f$ satisfies $V({\mathbf f})=\{\vec0\}$.  The homogeneous
  map ${\mathbf f}$ induces the Hopf map on any sphere, and so for any
  sphere $S^3(\vec0,R)$, there is an $\epsilon > 0$, such that there
  does not exist even a continuous ${\mathbf g}$, nowhere zero inside
  the sphere and $\epsilon$-close to ${\mathbf f}$ on $S^3(\vec0,R)$.
\end{example}
\begin{rem}\label{rem6.3}
  The Proofs in Section \ref{sec2} were not constructive, in that the
  notion of locally inessential merely asserts the existence of a
  homotopy, for example $\varphi$ in (\ref{eq0}), and then we appealed
  to the Weierstrass Approximation Theorem to get $\mathbf h$ in
  (\ref{eq32}).  It should be noted that finding polynomial
  representatives of homotopy classes is a difficult problem with a
  long history, see \cite{baum}, \cite{wood}.  In fact, one of the
  questions raised by \cite{baum} is on the existence of polynomial
  maps with isolated zeros.
\end{rem}


\begin{thebibliography}{Aikawa}

\bibitem[Aikawa]{aikawa} {\sc H.\ Aikawa}, {\sl H\"older continuity of
  the Dirichlet solution for a general domain}, Bull.\ London Math.\
  Soc.\ (6) {\bf34} (2002), 691--702.  MR 1924196 (2003f:31007).

\bibitem[A$_1$]{ankerthesis} {\sc D.\ Anker}, {\sl On Removing
  Isolated Zeroes of Vector Fields by Perturbation}, Ph.D.\ Thesis,
  University of Michigan, 1981.  MR 2631252.

\bibitem[A$_2$]{anker} {\sc D.\ Anker}, {\sl On removing isolated
  zeroes of vector fields by perturbation}, Nonlinear Analysis,
  Theory, Methods, \& Applications (9) {\bf8} (1984), 1005--1112.  MR
  0760200 (86d:58060).

\bibitem[B]{baum} {\sc P.\ Baum}, {\sl Quadratic maps and stable
  homotopy groups of spheres}, Illinois J.\ Math.\ {\bf 11} (1967),
  586--595.  MR 0220285 (36 \#3351).

\bibitem[BCR]{bcr} {\sc J.\ Bochnak}, {\sc M.\ Coste}, and {\sc
  M.\ Roy}, {\sl Real Algebraic Geometry}, MSM {\bf36}, Springer,
  1998.  MR 1659509 (2000a:14067).

\bibitem[BFGJ]{bfgj} {\sc R.\ Brown}, {\sc M.\ Furi}, {\sc
  L.\ G\'orniewicz}, and {\sc B.\ Jiang}, eds., {\sl Handbook of
  Topological Fixed Point Theory}, Springer, 2005.  MR 2170491
  (2006e:55001).

\bibitem[C$_1$]{ac} {\sc A.\ Coffman}, {\sl CR singular immersions of
complex projective spaces}, Beitr\"age zur Algebra und Geometrie (2)
{\bf43} (2002), 451--477.  MR 1957752 (2003k:32052).

\bibitem[C$_2$]{ac2} {\sc A.\ Coffman}, {\sl Real congruence of
complex matrix pencils and complex projections of real Veronese
varieties}, Linear Algebra Appl.\ {\bf370} (2003), 41--83.  MR 1994320
  (2004f:14026).

\bibitem[D$_1$]{d83} {\sc E.\ N.\ Dancer}, {\sl On the existence of zeros
  of perturbed operators}, Nonlinear Anal.\ (7) {\bf7} (1983),
  717--727.  MR 0707080 (84m:47077).

\bibitem[D$_2$]{d83b} {\sc E.\ N.\ Dancer}, {\sl Bifurcation under
  continuous groups of symmetries}, Systems of Nonlinear Partial
  Differential Equations (Oxford, 1982), 343--350, NATO ASI Ser.\ C
  {\bf111}, Reidel, Dordrecht, 1983.  MR0725531 (85f:58025).

\bibitem[D$_3$]{d84} {\sc E.\ N.\ Dancer}, {\sl Perturbation of zeros
  in the presence of symmetries}, J.\ Austral.\ Math.\ Soc.\ Ser.\ A
  (1) {\bf36} (1984), 106--125.  MR 0720004 (85g:58027).

\bibitem[Deloup]{d} {\sc F.\ Deloup}, {\sl The fundamental group of
  the circle is trivial}, The American Mathematical Monthly (5)
  {\bf112} (2005), 417--425.  MR 2139574 (2005k:57002).

\bibitem[E]{e2} {\sc A.\ Elgindi}, {\sl A topological obstruction to
  the removal of a degenerate complex tangent and some related
  homotopy and homology groups}, Internat.\ J.\ Math.\ (5) {\bf26}
  (2015), 1550025, 16 pp.  MR 3345506

\bibitem[F]{f} {\sc M.\ Fenille}, {\sl Epsilon Nielsen coincidence
  theory}, Cent.\ Eur.\ J.\ Math.\ (9) {\bf12} (2014), 1337--1348.  MR
  3201326.
  
\bibitem[G]{g} {\sc P.\ Garabedian}, {\sl Partial Differential
  Equations}, Wiley, 1964.  MR 0162045 (28 \#5247).

\bibitem[GT]{gt} {\sc D.\ Gilbarg} and {\sc N.\ Trudinger}, {\sl
  Elliptic Partial Differential Equations of Second Order}, Springer
  CIM, 2001.  MR 1814364 (2001k:35004).

\bibitem[LU]{lu} {\sc O.\ A.\ Ladyzhenskaya} and {\sc
  N.\ N.\ Ural'tseva}, {\sl Linear and Quasilinear Elliptic
  Equations}, translated from the 1964 Russian, Academic Press, 1968.
  MR 0244627 (39 \#5941).

\bibitem[L]{lee} {\sc J.\ Lee}, {\sl Introduction to Smooth
  Manifolds}, Second ed., GTM {\bf218}, Springer, 2013.  MR 2954043.

\bibitem[M]{m} {\sc V.\ G.\ Maz'ya}, {\sl Notes on H\"older regularity
  of a boundary point with respect to an elliptic operator of second
  order}, J.\ Math.\ Sci.\ (N.Y.) (4) {\bf196} (2014), 572--577;
  translated from Problemy Matematicheskogo Analiza {\bf74} (2013),
  117--121.  MR 3391313.

\bibitem[NNPV]{NNPV} {\sc M.\ Nestler}, {\sc I.\ Nitschke}, {\sc
  S.\ Praetorius}, and {\sc A.\ Voigt}, {\sl Orientational order on
  surfaces: the coupling of topology, geometry, and dynamics},
  J.\ Nonlinear Sci.\ (1) {\bf28} (2018), 147--191.  MR 3742799.

\bibitem[PP]{pp} {\sc J.\ Palis} and {\sc C.\ Pugh}, {\sl Fifty
  problems in dynamical systems}, in Dynamical Systems - Warwick 1974,
  345--353, LNM {\bf468}, Springer, 1975. MR 0391172 (52 \#11994).

\bibitem[ST]{st} {\sc C.\ Simon} and {\sc C.\ Titus}, {\sl Removing
  index-zero singularities with C$^1$-small perturbations}, in
  Dynamical Systems - Warwick 1974, 278--286, LNM {\bf468}, Springer,
  1975.  MR 0650643 (58 \#31254).

\bibitem[S]{s} {\sc E.\ Spanier}, {\sl Algebraic Topology},
McGraw-Hill, 1966.  MR 0210112 (35 \#1007).

\bibitem[W]{wood} {\sc R.\ Wood}, {\sl Polynomial maps from spheres to
  spheres}, Invent.\ Math.\ {\bf 5} (1968), 163--168. MR 0227999 (37
  \#3583).
\end{thebibliography}
\end{document}